\documentclass{amsart}

\usepackage{amsfonts, amsmath, amssymb, amsthm}
\usepackage{mathrsfs}
\usepackage{tikz-cd}

\newcommand{\inj}{\mathrm{inj}}

\newcommand{\vol}{\mathrm{vol}}

\newcommand{\st}{\,\big|\,}
\newcommand{\real}{\mathbb{R}}

\newcommand{\integer}{\mathbb{Z}}
\newcommand{\tor}{\mathbb{T}}

\newtheorem{theorem}{Theorem}[section]
\newtheorem{lemma}[theorem]{Lemma}
\newtheorem{corollary}[theorem]{Corollary}
\newtheorem{proposition}[theorem]{Proposition}
\newtheorem{remark}[theorem]{Remark}
\newtheorem{question}[theorem]{Question}

\numberwithin{equation}{section}

\begin{document}

\title{CENTRAL SPLITTING OF MANIFOLDS WITH NO CONJUGATE POINTS}

\author[JAMES DIBBLE]{JAMES DIBBLE}
\address{Department of Mathematics, University of Iowa, 14 MacLean Hall, Iowa City, IA 52242}
\email{james-dibble@uiowa.edu}


\date{}

\begin{abstract}
    Each compact Riemannian manifold with no conjugate points admits a family of functions whose integrals vanish exactly when central Busemann functions split linearly. These functions vanish when all central Busemann functions are sub- or superharmonic. When central Busemann functions are convex or concave, they must be totally geodesic. These yield generalizations of the splitting theorems of O'Sullivan and Eberlein for manifolds with no focal points and, respectively, nonpositive curvature.
\end{abstract}

\maketitle

\section{Introduction}

A key insight in Burago--Ivanov's proof that each Riemannian torus with no conjugate points is flat is that the asymptotic norm of its fundamental group $\pi_1(\tor^k) \cong \integer^k$ is Riemannian, in the sense that it is generated by an inner product \cite{BuragoIvanov1994}. In the case of an arbitrary compact Riemannian manifold $N$ whose fundamental group has center $Z(\pi_1(N))$ of rank $k$, O'Sullivan \cite{O'Sullivan1974} showed that, when $N$ has no focal points, it must be foliated by totally geodesic and flat $k$-dimensional toruses and, moreover, be covered by an isometric product with a flat $\tor^k$. This generalized a theorem of Wolf \cite{Wolf1964} about manifolds with nonpositive sectional curvature. It was later shown by Eberlein \cite{Eberlein1982} that each compact manifold with nonpositive sectional curvature and $Z(\pi_1(N))$ of rank $k$ is finitely covered by a diffeomorphic product with $\tor^k$. The aim of this paper is to explore conditions, between having no conjugate points and no focal points, that allow for variations on these results.

Associated to each $z_0,z_1 \in Z(\pi_1(N))$ is a function $F_{z_0 z_1} : N \to \real$ defined by \eqref{busemann laplacian integral}; the linear splitting of central Busemann functions on $\hat{N}$, where $\pi : \hat{N} \to N$ is the universal covering map, is one of a number of conditions equivalent to the vanishing of all $\int_N F_{z_0 z_1} \, d\vol_N$.

\begin{theorem}\label{linear splitting for subgroups}
    Let $N$ be a compact Riemannian manifold with no conjugate points and $Z$ a subgroup of $Z(\pi_1(N))$. Then the following are equivalent:\\
    (i) $\int_N F_{z_0 z_1} \, d\vol_N = 0$ for all $z_0,z_1 \in Z$;\\
    (ii) $B(z_0,z_1) = \frac{1}{\vol(N)} \int_N h(\omega_{z_0},\omega_{z_1}) \, d\vol_N$ for all $z_0,z_1 \in Z$;\\
    (iii) $h(\omega_{z_0}(x),\omega_{z_1}(x)) = B(z_0,z_1)$ for all $x \in N$ and all $z_0,z_1 \in Z$;\\
    (iv) $\omega_{\sum_i m_i z_i} = \sum_i m_i \omega_{z_i}$ for all $z_i \in Z$ and all $m_i \in \integer$.\\
\end{theorem}

\noindent The notation that appears in Theorem \ref{linear splitting for subgroups} is defined in the next two sections. The proof is partly an application of Green's identity.

It follows from the main result of \cite{DibbleToAppear} that any subgroup $Z$ of $Z(\pi_1(N))$ is isomorphic to $\integer^k$ for some $0 \leq k \leq \dim(N)$. When statements (i)-(iv) in Theorem \ref{linear splitting for subgroups} hold, one may define two integral formulations of an inner product that generates the asymptotic norm of $Z$. Moreover, a number of topological properties, which are known in the case of no focal points, hold (cf. \cite{O'Sullivan1974}).

\begin{theorem}\label{splitting for subgroups when integrals vanish}
    Let $N$ be a compact Riemannian manifold with no conjugate points and $Z$ a subgroup of $Z(\pi_1(N))$ for which statements (i)-(iv) in Theorem \ref{linear splitting for subgroups} hold. Then so do the following:\\
    (a) The asymptotic norm $\|\cdot\|_\infty$ of $Z$ with respect to any isomorphism $Z \to \integer^k$ is Riemannian;\\
    (b) If $w_1,\ldots,w_k$ generate $Z$ and $H_1,\ldots,H_k$ are corresponding horospheres, then $\hat{H} = \cap_{i=1}^k H_i$ is a simply connected submanifold of $\hat{N}$;\\
    (c) $\hat{N}$ is diffeomorphic to $\hat{H} \times \real^k$;\\
    (d) There exists a sequence of normal covering maps
    \[
        \hat{H} \times \real^k \xrightarrow{\psi_0} N_0 \times \tor^k \xrightarrow{\phi_0} N\textrm{,}
    \]
    with respective deck transformation groups $\pi_1(N_0) \times Z$ and $\Gamma$, such that $\psi_0$ is a product map, $N_0$ is orientable, $\pi_1(N_0)$ is a normal subgroup of $\pi_1(N)$ containing the commutator subgroup $[\pi_1(N), \pi_1(N)]$, and the sequences
    \[
        0 \to \pi_1(N_0) \times Z \to \pi_1(N) \to \Gamma \to 0
    \]
    and
    \[
        0 \to (\pi_1(N_0)/[\pi_1(N),\pi_1(N)]) \times Z \to H_1(N,\integer) \to \Gamma \to 0
    \]
    are exact.
\end{theorem}

\noindent Because $[\pi_1(N),\pi_1(N)] \subseteq \pi_1(N_0)$, the covering map $\phi_0$ is Abelian.

It is clear from \eqref{busemann laplacian integral} that the integrals $\int_N F_{z_0 z_1} \, d\vol_N$ vanish when central Busemann functions are harmonic. Lemma \ref{central busemann functions are harmonic} states that sub- or superharmonic central Busemann functions must be harmonic, which implies the following.

\begin{corollary}
    Let $N$ be a compact Riemannian manifold with no conjugate points and $Z$ a subgroup of $Z(\pi_1(N))$. If each Busemann function associated with $Z$ is sub- or superharmonic, then statements (i)-(iv) in Theorem \ref{linear splitting for subgroups} and the conclusions of Theorem \ref{splitting for subgroups when integrals vanish} hold.
\end{corollary}

\noindent Since they must be harmonic, concave or convex central Busemann functions must have vanishing Hessian and therefore be totally geodesic, in the sense that they map geodesics to geodesics. It is well known that Busemann functions are convex when $N$ has no focal points or, more narrowly, nonpositive sectional curvature \cite{Eschenburg1977}. Thus the following generalizes the splitting theorems of O'Sullivan and Eberlein.

\begin{theorem}\label{convex splitting for subgroups}
    Let $N$ be a compact Riemannian manifold with no conjugate points and $Z$ a subgroup of $Z(\pi_1(N))$. If each Busemann function associated with $Z$ is convex or concave, then, in addition to the conclusions of Theorem \ref{splitting for subgroups when integrals vanish}, the following hold:\\
    (a) $\hat{N}$ is isometric to $\hat{H} \times \real^k$, where $Z$ acts on each $\real^k$-fiber by translations;\\
    (b) $N$ is foliated by totally geodesic and flat $k$-dimensional toruses;\\
    (c) There exists a sequence of Riemannian covering maps
    \[
        \hat{H} \times \real^k \xrightarrow{\psi_1} N_1 \times \tor^k \xrightarrow{\phi_1} N\textrm{,}
    \]
    where $N_1$ is orientable, $\psi_1$ is a product, $\phi_1$ has finitely many sheets, and each map restricts on each $\real^k$- or $\tor^k$-fiber to a totally geodesic and locally isometric immersion onto a leaf of the $\tor^k$-foliation below.
\end{theorem}

\noindent It is not claimed in Theorem \ref{convex splitting for subgroups}(c) that $N_1 \times \tor^k$ has a product metric, and the statement cannot be strengthened to the isometric splitting of a finite cover, as there are examples in \cite{LawsonYau1972}, \cite{Eberlein1980}, and \cite{Eberlein1981} of compact $N_1 \times \tor^k$ with nonpositive curvature that are not finitely covered by isometric products with flat $k$-dimensional toruses.

Eberlein additionally proved that every compact manifold having nonpositive sectional curvature and fundamental group with nontrivial center is of a canonical form. This result also generalizes, provided one modifies the definition of a canonical manifold in \cite{Eberlein1982} appropriately. Let $\hat{H}$ be a complete and simply connected manifold with no conjugate points, $\Gamma_0$ a properly discontinuous group of isometries of $\hat{H}$, where $\hat{H}/\Gamma_0$ is compact, and $\rho : \Gamma_0 \to \tor^k$ a homomorphism whose kernel contains no nontrivial elements with fixed points. Define an action of $\Gamma_0$ on $\hat{H} \times \tor^k$ by $\Phi(\xi,\hat{x}) = (\Phi(\hat{x}),\rho(\Phi) \cdot \xi)$ for all $\Phi \in \Gamma_0$, $\xi \in \tor^k$, and $\hat{x} \in \hat{H}$. Then the quotient $(\hat{H} \times \tor^k)/\Gamma_0$ is called a \textbf{$k$-canonical} manifold.

\begin{theorem}\label{canonical form for subgroups}
    Let $N$ be a compact Riemannian manifold with no conjugate points and $Z$ a subgroup of $Z(\pi_1(N))$. If each Busemann function associated with $Z$ is convex or concave, then $N$ is a $k$-canonical manifold for $k = \mathrm{rank}\,Z$. When $Z = Z(\pi_1(N))$, $\Gamma_0$ is centerless.
\end{theorem}

\noindent The proof of Theorem \ref{canonical form for subgroups} will be omitted, as it exactly follows Eberlein's argument, using Theorem \ref{convex splitting for subgroups}(a) in the place of Lemma 1 of \cite{Eberlein1982}.

\section{Preliminaries}

Throughout this paper, $(N,h)$ will denote a compact, connected, $n$-dimensional, and $C^r$ Riemannian manifold for $r \geq 2$. The covering metric on $\hat{N}$ will be denoted by $\hat{h}$. The metric $h$ will have no conjugate points, which by definition means that the exponential map on each tangent space is nonsingular. It follows that, for each $p \in \hat{N}$, $\exp_p : T_p \hat{N} \to \real^n$ is a diffeomorphism. Thus $N$ is aspherical and, consequently, $\pi_1(N)$ is torsion free \cite{Hurewicz1936}. It is well known that a complete manifold with nonpositive sectional curvature has no focal points, that a complete manifold with no focal points has no conjugate points, and that neither of these implications is reversible within the space of compact manifolds of any fixed dimension at least two \cite{Gulliver1975}.

For any tangent vector $v$, denote by $\gamma_v$ the geodesic $t \mapsto \exp(tv)$. Corresponding to each unit vector $v \in T\hat{N}$ is the Busemann function $b_v : \hat{N} \to \real$ defined by $b_v(x) = \lim_{t \to \infty} \big[ t - d \big( \gamma_v(t), x \big) \big]$. It will be convenient to generalize the idea of a Busemann function to arbitrary tangent vectors by setting
\[
    b_v = \left\{ \begin{array}{cc} \|v\| b_{v/\|v\|} & \textrm{if } v \neq 0 \\ 0 & \textrm{if } v = 0 \end{array} \right.
\]
for any $v \in T\hat{N}$. It was essentially shown by Busemann \cite{Busemann1955} that associated to each $z \in Z(\pi_1(N))$ is a unique constant-length vector field $\omega_z$ on $N$ with the property that each $\gamma_{\omega_z(x)}|_{[0,1]}$ is a closed geodesic representing $z$ in $\pi_1(N,x)$. Denote by $\hat{\omega}_z$ the lift of $\omega_z$ to $\hat{N}$. A Busemann function $b_v$ is \textbf{central} if $v = \hat{\omega}_z(p)$ for some $z \in Z(\pi_1(N))$ and $p \in \hat{N}$. Following the arguments in \cite{BuragoIvanov1994} and \cite{Heber1994}, one finds that, for each $z \in Z(\pi_1(N))$, the corresponding central Busemann functions have gradient field $\hat{\omega}_z$. Arbitrary Busemann functions are known only to be $C^1$, except for manifolds with bounded asymptote, a class which includes those with no focal points, where they are $C^2$ \cite{Eschenburg1977}. However, for $z \in Z(\pi_1(N))$, the inverse function theorem implies that $\omega_z$ is $C^{r-1}$ and, consequently, that each central Busemann function is $C^r$ \cite{Dibble2019}.

For each $v \in T\hat{N}$, a horosphere of $v$ is, by definition, a level set of the Busemann function $b_v$. When $z \in Z(\pi_1(N))$, the horospheres of $\hat{\omega}_z(\hat{x})$ are the leaves of the normal distribution to $\hat{\omega}_z$ for any $\hat{x} \in \hat{N}$ and, consequently, form a $C^r$ foliation of $\hat{N}$ by hypersurfaces. In this way, one may speak of a horosphere $H$ of $z$ itself. For each such $H$, the normal bundle $NH$ is a trivial line bundle; since no point of $\hat{N}$ is focal to $H$, the exponential map on $NH$ is a diffeomorphism onto $\hat{N}$. In this way, $\hat{N} \cong H \times \real$.

The action of $\pi_1(N)$ on $\hat{N}$ by deck transformations will be denoted by $(\alpha,x) \mapsto \alpha(x)$. The following is a special case of an important lemma of Ivanov--Kapovitch \cite{IvanovKapovitch2014}. However, as they point out, the methods of Croke--Schroeder \cite{CrokeSchroeder1986} suffice to prove it for central elements. In particular, the narrow statement here requires only that the metric be $C^2$. The more general statement in \cite{IvanovKapovitch2014} requires $C^k$ regularity for some $k$ depending on $n$.

\begin{lemma}\label{constant change in busemann functions}
   Let $\alpha \in \pi_1(N)$ and $z \in Z(\pi_1(N))$. If $\gamma_v$ is an axis of $z$, then $b_v(\alpha(\hat{x})) - b_v(\hat{x})$ is independent of $\hat{x} \in \hat{N}$.
\end{lemma}

\noindent For any choices of $p,q \in \hat{N}$, $b_{\omega_{z_1}(p)}$ and $b_{\omega_{z_1}(q)}$ differ by a constant. Thus the function $B(z_0,z_1) = b_{\omega_{z_1}(p)}(z_0({x})) - b_{\omega_{z_1}(p)}(\hat{x})$, defined on $Z(\pi_1(N)) \times Z(\pi_1(N))$, is also independent of the choice of $p$. Loosely speaking, $B(z_0,z_1)$ is the change in the Busemann functions of $z_1$ in the direction of $z_0$.

Lemma \ref{constant change in busemann functions} implies, by the argument in \cite{IvanovKapovitch2014}, the virtual splitting of cyclic subgroups of $Z(\pi_1(N))$.

\begin{theorem}[Ivanov--Kapovitch]\label{virtual splitting}
    For each nontrivial $z \in Z(\pi_1(N))$, there exists a finite-index subgroup $G$ of $\pi_1(N)$ isomorphic to a direct product $G' \times \integer$, under which identification $z$ corresponds to a generator of the $\integer$-factor.
\end{theorem}

\noindent A simple consequence is that, when $Z(\pi_1(N))$ has rank at least $k$, a finite-index subgroup of $\pi_1(N)$ splits as a product with $\integer^k$.

\begin{corollary}\label{virtual splitting for higher rank}
    If $z_1,\ldots,z_k$ are independent elements of $Z(\pi_1(N))$, then there exists a finite-index subgroup $G$ of $\pi_1(N)$ isomorphic to a direct product $G' \times \integer^k$, under which identification the $\integer^k$-factor is generated by elements of the form $z_1^{m_1},\ldots,z_k^{m_k}$ for $m_i \geq 1$.
\end{corollary}

\begin{proof}
    Let $\varphi_i = (\pi_i,\sigma_i) : G_i \to G_i' \times \integer$ be the isomorphisms guaranteed by Theorem \ref{virtual splitting}, where each $G_i'$ is taken to be a subgroup of $G_i$ on which $\pi_i$ is the identity. Let $G = \cap_{i=1}^k G_i$, $G' = \cap_{i=1}^k G_i'$, and $\pi' = \pi_k \circ \cdots \circ \pi_1$. Then $\varphi = (\pi',\sigma_1,\ldots,\sigma_k) : G \to G' \times \integer^k$ is a homomorphism. Note that
    \[
        \mathrm{Ker}\,\varphi \subseteq \cap_{i=1}^k \mathrm{Ker}\,\sigma_i \subseteq \cap_{i=1}^k G_i' = G'\textrm{.}
    \]
    Since each $\pi_i$ is the identity on $G'$, $\varphi$ must be one-to-one. Since $G$ has finite index, there exist smallest $m_i \geq 1$ such that $z_i^{m_i} \in G$. It follows that $\varphi(G)$ is isomorphic to $G' \times \integer^k$ in such a way that the $z_i^{m_i}$ generate the $\integer^k$-factor.
\end{proof}

For any $z \in Z(\pi_1(N))$, write $\| z \|_\infty = B(z,z)^{1/2}$. If $Z$ is any subgroup of $Z(\pi_1(N))$, then $Z \cong \integer^k$ for some $0 \leq k \leq n$ \cite{DibbleToAppear}. With respect to a fixed isomorphism $Z \to \integer^k$, $\| \cdot \|_\infty$ agrees with the asymptotic norm of the orbit metric on $Z$ obtained from its action on $\hat{N}$. That is, if $d$ denotes the induced $\integer^k$-equivariant distance function on $\integer^k$, then, up to the identification of $Z$ with $\integer^k$,
\[
    \| z \|_\infty = \lim_{m \to \infty} \frac{d(0,mz)}{m}\textrm{,}
\]
and this extends uniquely to a norm on $\real^k$ (see Section 8.5 of \cite{BuragoBuragoIvanov2001} for details). Following \cite{BuragoIvanov1994}, one would like to define an inner product on $\real^k$ that extends $B$ with respect to this identification. In fact, elementary arguments show that $B$ satisfies many of the properties of an inner product: $B(z_0,mz_1) = mB(z_0,z_1) = B(z_0,mz_1)$ for all $m \in \integer$; by Corollary 4.2 of \cite{IvanovKapovitch2014}, $B(z_0 + z_1,z_2) = B(z_0,z_2) + B(z_1,z_2)$; and, significantly, $B$ satisfies the Cauchy--Schwarz inequality,
\begin{equation}\label{cauchy--schwarz}
|B(z_0,z_1)| \leq \|z_0\|_\infty \|z_1\|_\infty\textrm{,}
\end{equation}
with equality if and only if $z_0$ and $z_1$ are rationally related, in the sense that $m z_0 = \ell z_1$ for some $m,\ell \in \integer$ not both zero. However, it is not clear that, in general, $B$ is symmetric or additive in its second slot. For the purpose of showing that $\| \cdot \|_\infty$ is generated by an inner product, those two conditions are equivalent: Additivity would follow from symmetry, and, if additivity held, then the symmetrization of $B$ would extend to an inner product generating $\| \cdot \|_\infty$.

This section ends with a technical lemma about the horospheres of central Busemann functions. Although it is many ways unclear how well behaved they are, they may in some cases be used to construct fundamental domains of $\pi$. A preliminary lemma is first presented.

\begin{lemma}\label{primitive elements}
    If $z \in Z(\pi_1(N))$ is primitive, in the sense that $z \neq mz'$ for all $|m| > 1$ and $z' \in \pi_1(N)$, then, for each $\hat{x} \in \hat{N}$, $\pi \circ \gamma_{\hat{\omega}(\hat{x})}$ is injective on $[0,\|z\|_\infty)$.
\end{lemma}

\begin{proof}
    Assume that $\pi \circ \gamma_{\hat{\omega}(x)}(t_0) = \pi \circ \gamma_{\hat{\omega}(x)}(t_1)$ for $0 \leq t_0 < t_1 < \|z\|_\infty$. Write $T = t_1 - t_0$. Without loss of generality, replace $x$ with $\gamma_{\hat{\omega}(x)}(t_0)$, so that
    \[
        \pi \circ \gamma_{\hat{\omega}(x)}(0) = \pi \circ \gamma_{\hat{\omega}(x)}(T) = \pi \circ \gamma_{\hat{\omega}(x)}(\|z\|_\infty)\textrm{.}
    \]
    Since $\inj \big( \pi(x) \big) > 0$, $T = (a/b)\|z\|_\infty$ for some relatively prime $a,b \in \integer$ with $|b| > 1$. Note that $ma = nb + 1$ for some $m,n \in \integer$. Since $\pi \circ \gamma_{\hat{\omega}(x)}((ma/b)\|z\|_\infty) = \pi \circ \gamma_{\hat{\omega}(x)}(0)$, $\gamma_{\omega(\pi(x))}|_{[0,(1/b)\|z\|_\infty]}$ is a closed geodesic that represents an element $w$ of $\pi_1(N)$ with $z = bw$, which is a contradiction.
\end{proof}

\begin{lemma}\label{fundamental domain}
    Let $z \in Z(\pi_1(N))$ be primitive and $H$ a horosphere of $z$. Then there exist disjoint open subsets $U_1,\ldots,U_K$ of $H$ such that, for $V_i = U_i \times [0,\|z\|_\infty)$ and $V = \cup_{i=1}^K V_i$, $\pi$ is injective on $V$ and $\vol_h \big( N \setminus \pi(V) \big) = 0$.
\end{lemma}

\begin{proof}
    It follows from $\|z\|_\infty$-periodicity that there exists $\varepsilon > 0$ such that, for each $x \in H$, the exponential map is injective on $B(x,\varepsilon) \times [0,\|z_0\|_\infty) \subseteq N H$. By compactness, one may choose a finite subset $\{ W_1,\ldots,W_K \}$ of $\{ B(x,\varepsilon) \st x \in H \}$ with the property that $\{ \pi(W_i \times [0,\|z_0\|_\infty]) \}$ is an open cover of $N$. Let $U_1 = W_1$. The remaining $U_i$ are defined inductively: Suppose $U_1,\ldots,U_{i-1}$ have been defined and satisfy $U_j \subseteq W_j$. For each $1 \leq j \leq i-1$, there exist at most finitely many $\alpha_{j,1},\ldots,\alpha_{j,M_j} \in \pi_1(N)$ such that $\alpha_{j,m}(\hat{U}_j \times [0,\|z\|_\infty]) \cap W_i \neq \emptyset$. Let
    \[
        U_i = W_i \setminus \bigcup_{j=1}^{i-1} \bigcup_{m=1}^{M_j} \alpha_{j,m}(\hat{U}_j \times [0,\|z\|_\infty])\textrm{.}
    \]
    The sets $U_1,\ldots,U_K$ constructed in this way have the desired properties.
\end{proof}

\section{Integral formulations of $B$}

Let $z_0,z_1 \in Z(\pi_1(N))$. The function $F_{z_0 z_1} : N \to \real$ in Theorem \ref{splitting for subgroups when integrals vanish} will be defined by an equivariant construction on $\hat{N}$. Denote by $\Delta$ the Laplace--Beltrami operator. Fix $\hat{p} \in \hat{N}$ and $v_i = \hat{\omega}_{z_i}(\hat{p})$. For $x \in N$, let $\hat{x} \in \pi^{-1}(x)$ and $v = \hat{\omega}_{z_0}(\hat{x})$, and define $F_{z_0 z_1}(x)$ to be the average value of $b_{v_1} \Delta b_{v_0}$ along $\gamma_v|_{[0,1]}$ with respect to arclength:
\begin{equation}\label{busemann laplacian integral}
    F_{z_0 z_1}(x) = \oint_{\gamma_v|_{[0,1]}} b_{v_1} \Delta b_{v_0} ds\textrm{.}
\end{equation}
More specifically, $F_{z_0 z_1} = 0$ if $z_0$ is the identity, and
\[
    F_{z_0 z_1}(x) = \frac{1}{\|z_0\|_\infty} \int_{\gamma_v|_{[0,1]}} b_{v_1} \Delta b_{v_0} ds
\]
otherwise. To see that $F_{z_0 z_1}$ is well defined, first recall the following classical result.

\begin{lemma}[Green's identity]\label{green's identity}

Let $(M,g)$ be a $C^2$ Riemannian manifold with boundary $\partial M$ and $\nu$ an outward-pointing unit normal vector field along $\partial M$. If $\phi,\psi : M \to \real$ are $C^2$ functions, then
\[
\int_M \phi \Delta \psi \, d\vol_M + \int_M g(\nabla \phi,\nabla \psi) \, d\vol_M = \int_{\partial M}  \phi g(\nabla \psi, \nu) \, d\vol_{\partial M}\textrm{.}
\]
\end{lemma}

\noindent This generalizes in a straightforward way to manifolds with corners. A simple consequence is that, for a central Busemann function $b_v$, $\Delta b_v$ measures the change in volume of its horospheres under its gradient flow.

\begin{lemma}\label{derivative of area}
     Let $z \in Z(\pi_1(N))$ be nontrivial. Let $b_v$ be a corresponding central Busemann function and $T \in \real$. For each $t \in \real$, denote by $H_t$ the horosphere of $z$ along which $b_v = t$. Let $U \subset H_T$ be any open set with $C^r$ boundary and compact closure, and let $U_t = U \times [T,T+t]$ with respect to the splitting $\hat{N} \cong H_T \times \real$. Then
    \[
        \frac{d}{dt} \vol_{H_{T+t}}(U \times \{ T + t \}) = \frac{1}{\|z\|_\infty} \int_{U \times \{ T + t \}} \Delta b_v \, d\vol_{H_{T+t}}\textrm{.}
    \]
\end{lemma}

\begin{proof}
    Denote by $\nu$ the outward-pointing unit normal vector field along the $C^r$ portion of $\partial U_t$, so that, except at corners, $\nu = -\hat{\omega}_z/\|z\|_\infty$ along $U \times \{ T \}$, $\nu = \hat{\omega}_z/\|z\|_\infty$ along $U \times \{ T + t \}$, and $\hat{h}(\nu,\hat{\omega}_z) = 0$ along $\partial U \times [0,\|z\|_\infty]$. By Green's identity,
    \begin{align*}
        \int_{U_t} \Delta b_v \, d\vol_{\hat{N}} &= \int_{\partial U_t} \hat{h}(\hat{\omega}_z,\nu) \, d\vol_{\partial U_t}\\
        &= \|z\|_\infty [ \vol_{H_{T+t}}(U \times \{ T + t \}) - \vol_{H_T}(U \times \{ T \}) ]\textrm{.}
    \end{align*}
    It follows from the coarea formula that
    \[
        \int_{U_t} \Delta b_v \, d\vol_{\hat{N}} = \int_T^{T+t} \int_{U \times \{ T + s \}} \Delta b_v d \vol_{H_{T+s}} ds\textrm{.}
    \]
    The result follows immediately.
\end{proof}

\noindent If $\hat{x}' \in \pi^{-1}(x)$ and $v' = \hat{\omega}_{z_0}(\hat{x}')$, then $\hat{x}' = \alpha(x)$ for some $\alpha \in \pi_1(N)$. Note that, for all $t$, $\Delta b_{v_0}(\gamma_v (t)) = \Delta b_{v_0}(\gamma_{v'}(t))$ and, by Lemma \ref{constant change in busemann functions}, $b_{v_1}(\gamma_v (t)) - b_{v_1}(\gamma_{v'}(t)) = B(\alpha,z_1) \in \real$. Moreover, for $T = b_{v_0}(\hat{x})/{\|z_0\|_\infty^2}$, $U^\varepsilon = B(\gamma_v(-T),\varepsilon) \subseteq H_0$, and $U_{\|z_0\|_\infty}^\varepsilon = U^\varepsilon \times [T,T+\|z_0\|_\infty]$, an application of Lemma  \ref{derivative of area} shows that
\begin{align*}
    \int_{\gamma_v|_{[0,1]}} \Delta b_{v_0} dt &= \lim_{\varepsilon \to 0} \int_{U^\varepsilon_{\|z_0\|_\infty}} \Delta b_{v_0} \, d\vol_{\hat{N}}\\
    &= \lim_{\varepsilon \to 0} \|z_0\|_\infty [ \vol_{H_{T+\|z_0\|_\infty}}(U^\varepsilon \times \{ T+\|z_0\|_\infty \}) - \vol_{H_T}(U^\varepsilon \times \{ T \}) ]\\
    &= 0\textrm{.}
\end{align*}
Thus
\begin{align*}
    \int_{\gamma_{v'}|_{[0,1]}} b_{v_1} \Delta b_{v_0} dt &= \int_{\gamma_v|_{[0,1]}} [b_{v_1} + B(\alpha,z_1)] \Delta b_{v_0} ds\\
    &= \int_{\gamma_v|_{[0,1]}} b_{v_1} \Delta b_{v_0} ds + B(\alpha,z_1) \int_{\gamma_v|_{[0,1]}} \Delta b_{v_0} ds\\
    &= \int_{\gamma_v|_{[0,1]}} b_{v_1} \Delta b_{v_0} ds\textrm{.}
\end{align*}
So $F_{z_0 z_1}$ is well defined. Similarly, $F_{z_0 z_1}$ is independent of the choice of $\hat{p}$, as replacing $\hat{p}$ with $\hat{q}$ induces a constant change in $b_{v_1}$ of $b_{v_1}(\hat{q}) - b_{v_1}(\hat{p})$ while leaving $\Delta b_{v_0}$ unchanged.

For any $\hat{p}$ as above, let $S$ be a fundamental domain of $\pi$ constructed using a horosphere $H_T$ of $v_0$ as in Lemma \ref{fundamental domain}, and write $S_t = S \times [t,t+\|z_0\|_\infty]$ for each $t \in [T,T+\|z_0\|_\infty]$. Then
\begin{equation}\label{busemann laplacian integral over N}
\begin{aligned}
    \int_N F_{z_0 z_1} \, d\vol_N &= \int_{S_T} \int_{\gamma_{\hat{\omega}_{z_0}(\hat{x})}|_{[0,1]}} b_{v_1} \Delta b_{v_0} ds \,\, d\vol_{\hat{N}}\\
    &= \int_T^{T+\|z_0\|_\infty} \int_{S \times \{ t \}} \int_{\gamma_{\hat{\omega}_{z_0}(\hat{x})}|_{[0,1]}} b_{v_1} \Delta b_{v_0} ds \,\, d\vol_{H_t} \,dt\\
    &= \int_T^{T+\|z_0\|_\infty} \int_{S_t} b_{v_1} \Delta b_{v_0} \, d\vol_{\hat{N}} \,dt\textrm{.}
\end{aligned}
\end{equation}

\begin{theorem}\label{green's identity integral}
    Let $z_0,z_1 \in Z(\pi_1(N))$. Then
    \[
        B(z_0,z_1) = \frac{1}{\vol(N)} \int_N [ h(\omega_{z_0},\omega_{z_1}) + F_{z_0 z_1} ] \, d\vol_N\textrm{.}
    \]
\end{theorem}

\begin{proof}
    By Theorem \ref{virtual splitting}, there exist a finite covering map $\psi : \tilde{N} \to N$, with covering metric $\tilde{h}$, and a primitive $\tilde{z}_0 \in Z \big( \pi_1(\tilde{N}) \big)$ such that $\psi_*(\tilde{z}_0) = z_0$. For each $i = 1,2$, denote by $\tilde{\omega}_i$ the lift of $\omega_i$ to $\tilde{N}$. Then
    \[
        \frac{1}{\vol(N)} \int_N h(\omega_{z_0},\omega_{z_1}) \, d\vol_N = \frac{1}{\vol_{\tilde{h}}(\tilde{N})} \int_{\tilde{N}} \tilde{h}(\tilde{\omega}_0,\tilde{\omega}_1) d \vol_{\tilde{h}}\textrm{.}
    \]
    Fix $\hat{x} \in \hat{N}$ and $v_i = \hat{\omega}_{z_i}(\hat{x})$. Let $T \in \real$. For the horospheres $H_t$ of $z_0$, let $U_1,\ldots,U_K$ be the open subsets of $H_T$ obtained by applying Lemma \ref{fundamental domain} to the covering $\hat{N} \to \tilde{N}$, and write $S = \cup_{i=1}^K U_i$ and $S_t = S \times [t,t+\|z\|_\infty]$. If each $U_i$ has $C^r$ boundary, then Green's identity shows that
    \begin{equation}\label{green's identity for h}
        \int_{S_t} \hat{h}(\hat{\omega}_{z_0},\hat{\omega}_{z_1}) \, d\vol_{\hat{N}} = \int_{\partial S_t} b_{v_1} \hat{h}(\hat{\omega}_{z_0},\nu) \, d\vol_{\partial S_t} - \int_{S_t} b_{v_1} \Delta b_{v_0} \, d\vol_{\hat{N}}\textrm{.}
    \end{equation}
    Applying Lemma \ref{constant change in busemann functions}, one has that
    \begin{align*}
        \int_{\partial S_t} b_{v_1} \hat{h}(\hat{\omega}_{z_0},\nu) d \vol_{\partial S_t} &= \|z_0\|_\infty \big[ \int_{S \times \{ t + \|z_0\|_\infty \}} b_{v_1} d \vol_{H_{\|z_0\|_\infty}} - \int_{S \times \{ t \}} b_{v_1} d \vol_{H_0} \big]\\
        &= \|z_0\|_\infty B(z_0,z_1)\vol_{H_t}(S \times \{ t \})\textrm{.}
    \end{align*}
    Suppose that $z_0$ is not the identity, as the result holds otherwise. Substituting the above into equation \eqref{green's identity for h}, integrating both sides from $T$ to $T + \|z_0\|_\infty$, and using the fact that $\int_{S_t} \hat{h}(\hat{\omega}_{z_0},\hat{\omega}_{z_1}) \, d\vol_{\hat{N}} =  \int_{\tilde{N}} \tilde{h}(\tilde{\omega}_0,\tilde{\omega}_1) d \vol_{\tilde{h}}$ yields
    \begin{equation}
    \begin{aligned}
        B(z_0,z_1) &= \frac{1}{\vol(N)} \int_N h(\omega_{z_0},\omega_{z_1}) \, d\vol_N\\
        &\qquad \qquad \quad+ \frac{1}{\|z_0\|_\infty \vol(N)} \int_T^{T+\|z_0\|_\infty} \int_{S_t} b_{v_1} \Delta b_{v_0} \, d\vol_{\hat{N}} dt\\
        &= \frac{1}{\vol(N)} \int_N h(\omega_{z_0},\omega_{z_1}) \, d\vol_N + \frac{1}{\vol(N)} \int_N F_{z_0 z_1} \, d\vol_N\textrm{,}
    \end{aligned}
    \end{equation}
    the latter equality following from equation \eqref{busemann laplacian integral over N}. In the general case, one may apply a similar argument to a union of open sets $U_{i,m} \subseteq U_i$ that have $C^r$ boundary and whose measures converge to that of $U_i$ as $m \to \infty$.
\end{proof}

\begin{corollary}\label{integrate with respect to h}
    Let $z_0,z_1 \in Z(\pi_1(N))$. Then the following hold:\\
    (a) $B(z_0,z_1) = B(z_1,z_0)$ if and only if $\int_N F_{z_0 z_1} \, d\vol_N = \int_N F_{z_1 z_0} \, d\vol_N$;\\
    (b) $B(z_0,z_1) = \frac{1}{\vol(N)} \int_N h(\omega_{z_0},\omega_{z_1}) \, d\vol_N$ if and only if $\int_N F_{z_0 z_1} \, d\vol_N = 0$.
\end{corollary}

Let $Z$ be any subgroup of $Z(\pi_1(N))$ of rank $k$, and fix an isomorphism $D : Z \to \integer^k$. Denote by $e_1,\ldots,e_k$ the standard basis for $\real^k$, and let $w_i = D^{-1}(e_i)$. Let $f : \tor^k \to N$ be any fixed $C^1$ map satisfying $f_*(\alpha_i) = w_i$, where $\alpha_1,\ldots,\alpha_k$ are generators for $\pi_1(\tor^k)$. Since $N$ is aspherical, and therefore an Eilenberg--Mac Lane space, such maps are guaranteed to exist, although in the case of no conjugate points they may be constructed more concretely by iteratively exponentiating around loops (see Section 3.3 of \cite{Dibble2019}).

\begin{theorem}\label{integrate with respect to g}
    Let $z_0,z_1 \in Z$ and $\beta_1 = f^*(\omega_{z_1}^\flat)$ (i.e., $\beta_1(v) = h \big( \omega_{z_1},f_*(v) \big)$ for all $v \in T \tor^k$). Then $B(z_0,z_1) = \int_{\tor^k} \beta_1 \circ D(z_0) \, d\vol_g$ for the standard flat metric $g$ on $\tor^k$.
\end{theorem}

\begin{proof}
    Suppose $k \geq 2$, as otherwise the result is clear. Denote by $\phi : \real^k \to \tor^k$ the covering map that quotients by $\integer^k$. Let $V = D(z_0) = \sum_{i=1}^m v_i e_i$ be a constant vector field on $\tor^k$. Suppose $V \neq 0$, so that some $v_i \neq 0$. Let $u_i = V$ and, for each $j \neq i$, let $u_j$ be the vector in $\real^k$ whose only nonzero entries are $-v_j$ in the $i$-th component and $v_i$ in the $j$-th component. Let $P$ be the parallelepiped in $\real^k$ determined by $u_1,\ldots,u_k$. Then $P$ is the union of $v_i^{k-2} \sum_{j=1}^k v_j^2 = v_i^{k-2}\|V\|_{\real^k}^2$ fundamental domains of $\phi$. Denote by $Q$ the face of $P$ that contains the origin but not $u_i$, and, for each $p \in Q$, denote by $\alpha_p$ the geodesic $t \mapsto p + tV/\|V\|_{\real^k}$.

    The map $f$ lifts to a map $F : \real^k \to \hat{N}$ such that $f \circ \phi = \pi \circ F$. Let $\hat{V} = \phi^{-1}(V)$, so that
    \[
        h \big( \omega_{z_1},f_*(V) \big) = \hat{h} \big( \hat{\omega}_{z_1},F_*(\hat{V}) \big) = \hat{h} \big( \hat{\omega}_{z_1}, F_*(\alpha_x'(t)) \big)\textrm{.}
    \]
    The coarea formula implies that
    \begin{align*}
        \int_{\tor^k} h \big(  \omega_{z_1},f_*(V) \big) d \vol_g &= \frac{1}{v_i^{k-2}\|V\|_{\real^k}} \int_Q \Big[ \int_0^{\|V\|_{\real^k}} \hat{h} \big( \hat{\omega}_{z_1}, F_*(\alpha_x'(t)) \big) dt \Big] d \vol_Q\\
        &= \frac{1}{v_i^{k-2}\|V\|_{\real^k}} \int_Q [b_{z_1} \circ F(x + V) \big) - b_{z_1} \circ F(x)] d \vol_Q\\
        &= \frac{1}{v_i^{k-2}\|V\|_{\real^k}} \int_Q [b_{z_1} \big( z_0(F(x)) \big) - b_{z_1}(F(x))] d \vol_Q\\
        &= \frac{1}{v_i^{k-2}\|V\|_{\real^k}} \int_Q B(z_0,z_1) d \vol_Q\\
        &= B(z_0,z_1)\textrm{.}
    \end{align*}
    If $V = 0$, then $z_0$ is the identity, and one obtains the same equality.
\end{proof}

\begin{remark}
    For $N = \tor^k$ and $z \in \integer^k$, $z = D(\omega_z) = \int_{\tor^k} \omega_z \, d\vol_g$, where $g$ is the standard flat metric on $\tor^k$. This continuously extends the direction at infinity in \cite{BuragoIvanov1994} to the leaves of the Heber foliation \cite{Heber1994} of the unit sphere bundle of $\tor^k$.
\end{remark}

\section{Proof of Theorems \ref{linear splitting for subgroups} and \ref{splitting for subgroups when integrals vanish}}

As before, let $Z$ be a subgroup of $Z(\pi_1(N))$, fix an isomorphism $D : Z \to \integer^k$, and let $w_i = D^{-1}(e_i)$, so that $w_1,\ldots,w_k$ generate $Z$. Write $\mathscr{B}_0 = \{ \omega_z \st z \in Z \}$, and denote by $\mathscr{B}$ the vector space of finite formal linear combinations of elements of $\mathscr{B}_0$ with real coefficients. (For each $\hat{x} \in \hat{N}$, $\mathscr{B}$ may be identified with the space of finite formal combinations of Busemann functions in $Z$ that vanish at $\hat{x}$.) Roughly speaking, one may extend $D$ to a linear direction at infinity $\mathscr{D} : \mathscr{B} \to \real^k$ by setting $\mathscr{D}(\sum_i a_i \omega_{z_i}) = \sum_i a_i D(z_i)$. In the other direction, there is the inclusion $\iota : \integer^k \to \mathscr{B}$ defined by $\iota(\sum_i m_i e_i) = \omega_{\sum_i m_i w_i}$.

For each $\zeta_1,\zeta_2 \in \mathscr{B}$ (i.e., $\zeta_i = \sum_j a_{ij} \omega_{z_{ij}}$ for $z_{ij} \in Z$ and $a_{ij} \in \real$), let $\beta_i = f^*(\zeta_i^\flat)$. Define
\[
    G(\zeta_1,\zeta_2) = \frac{1}{2} \int_{\tor^k} [ \beta_1 \circ \mathscr{D}(\zeta_2) + \beta_2 \circ \mathscr{D}(\zeta_1) ] d\mu_g\textrm{.}
\]
It is clear that $G$ is symmetric and bilinear. By Corollary \ref{integrate with respect to h}(a) and Theorem \ref{integrate with respect to g}, when (i) holds,
\[
    G(\omega_{z_0},\omega_{z_1}) = B(z_0,z_1) = B(z_1,z_0)\textrm{.}
\]
At the same time, one may define a semi-inner product $H$ on $\mathscr{B}$ by
\[
    H(\zeta_1,\zeta_2) = \frac{1}{\vol(N)} \int_N h(\zeta_1,\zeta_2) \, d\vol_N\textrm{.}
\]
Again assuming (i), Corollary \ref{integrate with respect to h}(b) implies that
\[
    H(\omega_{z_0},\omega_{z_1}) = B(z_0,z_1) = B(z_1,z_0)\textrm{.}
\]
In this case, since they agree on $\mathscr{B}_0 \times \mathscr{B}_0$, $G$ and $H$ define the same semi-inner product on $\mathscr{B}$ and, consequently, the same semi-norm. Overloading notation, write $\|\cdot\|_\infty = \|\cdot\|_G = \|\cdot\|_H$.

\begin{lemma}\label{constant inner products}
    Let $Z$ be a subgroup of $Z(\pi_1(N))$ such that $\omega_{\sum_i m_i z_i} = \sum_i m_i \omega_{z_i}$ for all $z_i \in Z$ and all $m_i \in \integer$. Then, for all $z_0,z_1 \in Z$, the following hold:\\
    (a) $[\omega_{z_0},\omega_{z_1}] = 2\nabla_{\omega_{z_0}} \omega_{z_1}$ (i.e., $\nabla_{\omega_{z_0}} \omega_{z_1} = - \nabla_{\omega_{z_1}} \omega_{z_0}$);\\
    (b) $h([\omega_{z_0},\omega_{z_1}],\omega_{z_0}) = h([\omega_{z_0},\omega_{z_1}],\omega_{z_1}) = 0$;\\
    (c) $h(\omega_{z_0}(x),\omega_{z_1}(x)) = B(z_0,z_1)$ for all $x \in N$.
\end{lemma}

\begin{proof}
    One has that
    \begin{align*}
        0 &= \nabla_{\omega_{z_0+z_1}} \omega_{z_0+z_1} = \nabla_{\omega_{z_0}} \omega_{z_0} + \nabla_{\omega_{z_0}} \omega_{z_1} + \nabla_{\omega_{z_1}} \omega_{z_0} + \nabla_{\omega_{z_1}} \omega_{z_1}\\
        &= \nabla_{\omega_{z_0}} \omega_{z_1} + \nabla_{\omega_{z_1}} \omega_{z_0}\textrm{,}
    \end{align*}
    so $\nabla_{\omega_{z_0}} \omega_{z_1} = -\nabla_{\omega_{z_1}} \omega_{z_0}$. This proves (a). Moreover,
    \[
        0 = \omega_{z_1} [h(\omega_{z_0},\omega_{z_0})] = 2h(\nabla_{\omega_{z_1}} \omega_{z_0},\omega_{z_0}) = -2h(\nabla_{\omega_{z_0}} \omega_{z_1},\omega_{z_0}) = -2\omega_{z_0}[h(\omega_{z_1},\omega_{z_0})]\textrm{.}
    \]
    One may deduce (b) from the first three equalities above. Let $\hat{x} \in \hat{N}$, and write $\gamma = \gamma_{\hat{\omega}_{z_0}(\hat{x})}$. Since $h(\hat{\omega}_{z_1},\hat{\omega}_{z_0})$ is constant along $\gamma$, one finds that
    \begin{align*}
        B(z_0,z_1) &= b_{z_1} \big( z_0(x) \big) - b_{z_1}(x)\\
        &= \int_0^1 h \big( (\hat{\omega}_{z_0} \circ \gamma)(t), (\hat{\omega}_{z_1} \circ \gamma)(t) \big) dt\\
        &= h \big( \hat{\omega}_{z_0}(x),\hat{\omega}_{z_1}(x) \big)\textrm{,}
    \end{align*}
    which proves (c).
\end{proof}

\noindent It is now possible to prove Theorem \ref{linear splitting for subgroups}. Statements (i) and (ii) are equivalent by Corollary \ref{integrate with respect to h}(b). By Lemma \ref{constant inner products}(c), (iv) implies (iii), and it is clear that (iii) implies (ii). Therefore, the proof can be completed by showing that (i) implies (iv).

Suppose (i) holds. Write $\alpha = \omega_{\sum_i m_i z_i}$, $\beta = \sum_i m_i \omega_{z_i}$, and $\zeta = \alpha - \beta$. Then
\begin{align*}
    \| \zeta \|_\infty^2 &= G(\zeta,\zeta)\\
    &= \int_{\tor^k} h \big( \zeta, D(\zeta) \big) d \vol_g\\
    &= \int_{\tor^k} h \big( \zeta, 0 \big) d \vol_g\\
    &= 0\textrm{,}
\end{align*}
which means that
\begin{align*}
    0 &= \vol(N) H(\zeta,\zeta)\\
    &= \int_N h(\zeta,\zeta) d \vol_h\\
    &= \int_N [\|\alpha\|_h^2 + \|\beta\|_h^2 - 2h(\alpha,\beta) ] d \vol_h\textrm{.}
\end{align*}
Thus
\begin{equation}\label{cs-equality}
    \int_N h(\alpha,\beta) d \vol_h = \int_N \frac{1}{2}(\|\alpha\|_h^2 + \|\beta\|_h^2) d \vol_h\textrm{.}
\end{equation}
By the Cauchy--Schwarz inequality,
\[
    h(\alpha,\beta) \leq \|\alpha\|_h \|\beta\|_h\textrm{,}
\]
so
\[
    \int_N (\|\alpha\|_h^2 + \|\beta\|_h^2) d \vol_h \leq \int_N 2 \|\alpha\|_h^2 \|\beta\|_h^2 d \vol_h\textrm{.}
\]
At the same time, $2 \|\alpha\|_h \|\beta\|_h \leq \|\alpha\|_h^2 + \|\beta\|_h^2$. Thus
\[
    \int_N 2\|\alpha\|_h^2 \|\beta\|_h^2 d \vol_h = \int_N (\|\alpha\|_h^2 + \|\beta\|_h^2) d \vol_h\textrm{,}
\]
which implies that $\|\alpha\|_h = \|\beta\|_h$ on $N$. Substituting into equation \eqref{cs-equality}, one finds that
\[
    \int_N h(\alpha,\beta) d \vol_h = \int_N \|\alpha\|_h^2 \|\beta\|_h^2 d \vol_h
\]
and, consequently, that $h(\alpha,\beta) = \|\alpha\|_h \|\beta\|_h$ on $N$. It follows that $\alpha = \beta$, which is (iv). This proves Theorem \ref{linear splitting for subgroups}.

Theorem \ref{splitting for subgroups when integrals vanish}(a) is a consequence of the following lemma and the fact that, when (i)-(iv) hold, $G$ agrees with $H$.
\begin{lemma}\label{asymptotic norm is riemannian}
    Let $Z$ be a subgroup of $Z(\pi_1(N))$ of rank $k$. If $G$ is positive semi-definite, then
    \[
        G(\iota(z_0),\iota(z_1)) = G(\omega_{z_0},\omega_{z_1}) = \frac{1}{2}[B(z_0,z_1) + B(z_1,z_0)]
    \]
    extends to an inner product on $\real^k$ that induces the asymptotic norm $\|\cdot\|_\infty$ of $Z$ with respect to the isomorphism $D : Z \to \integer^k$.
\end{lemma}

\begin{proof}
    Since $G$ is positive semi-definite, it is a semi-inner product and induces a semi-norm $\|\cdot\|_G$ on $\mathscr{B}$. If $\zeta \in \mathrm{Ker}(\mathscr{D})$, then it follows from the definition of $G$ that $\|\zeta\|_G = 0$. Conversely, suppose $\|\zeta\|_G = 0$. Since $\| \iota(\cdot) \|_G = \| \cdot \|_\infty$, there exists $c > 0$ such that $\| \iota(x) \|_G \geq c\|x\|_{\real^k}$ for all $x \in \integer^k$. Write $W = \mathrm{span}\, \{ \omega_{w_1},\ldots,\omega_{w_k} \}$. Since $\mathscr{D}|_W : W \to \real^k$ is an invertible linear map, there exists $C > 0$ such that $\|w\|_G \leq C\|\mathscr{D}(w)\|_{\real^k}$ for all $w \in W$. There exist $K_i \to \infty$ and $v_i \in W$ with $\| \mathscr{D}(v_i) \|_{\real^k} \leq 1$ such that $\mathscr{D}(K_i \zeta + v_i) \in \integer^k$. Since $\mathscr{D} \circ \iota$ is the identity map on $\integer^k$, one has that
    \[
        \mathscr{D} \big( K_i \zeta + v_i - \iota \circ \mathscr{D}(K_i \zeta + v_i) \big) = 0\textrm{,}
    \]
    and, consequently,
    \[
        \big| \| K_i \zeta + v_i \|_G - \| \iota \circ \mathscr{D}(K_i \zeta + v_i) \|_G \big| \leq \| K_i \zeta + v_i - \iota \circ \mathscr{D}(K_i \zeta + v_i) \|_G = 0\textrm{.}
    \]
    It follows that
    \[
        c \| \mathscr{D}(K_i \zeta + v_i) \|_{\real^k} \leq \| \iota \circ \mathscr{D}(K_i \zeta + v_i) \|_G = \|K_i \zeta + v_i\|_G \leq \|K_i \zeta\|_G + \|v_i\|_G \leq C\textrm{.}
    \]
    Therefore, $\| \mathscr{D}(\zeta) \|_{\real^k} \leq (1 + C/c)/K_i \to 0$ as $i \to \infty$ and, consequently, $\mathscr{D}(\zeta) = 0$. Thus $\mathrm{Ker}(\mathscr{D}) = \{ \zeta \in \mathscr{B} \st \| \zeta \|_G = 0 \}$.

    Let $\tilde{\mathscr{B}}$ be the vector space obtained as the quotient of $\mathscr{B}$ by $\mathrm{Ker}(\mathscr{D})$, $\tilde{G}$ the corresponding inner product on $\mathscr{B}$, and $\| \cdot \|_{\tilde{G}}$ the induced norm. The map $\mathscr{D}$ descends to a map $\tilde{\mathscr{D}}$ on $\tilde{\mathscr{B}}$ with trivial kernel. By construction, $\tilde{\mathscr{D}}([\omega_{w_i}]) = e_i$ for each $i$, so $\tilde{\mathscr{D}}$ is surjective. Thus $\tilde{\mathscr{D}}$ is a linear isomorphism, and $\tilde{\mathscr{D}}^{-1}(\integer^k) = \{ \sum_i m_i [\omega_{w_i}] \st m_i \in \integer \} = \{ [\omega_z] \st z \in Z \}$ projects to a dense subset of the unit sphere in $\tilde{\mathscr{B}}$. Since
    \[
        \| [\omega_z] \|_{\tilde{G}}^2 = \tilde{G}([\omega_z],[\omega_z]) = G(\omega_z,\omega_z) = B(z,z) = \|z\|_\infty^2 = \| \tilde{\mathscr{D}}([\omega_{z}]) \|_\infty^2
    \]
    for all $z \in Z$, it follows by continuity that $\tilde{\mathscr{D}} : (\tilde{\mathscr{B}},\|\cdot\|_{\tilde{G}}) \to (\real^k,\|\cdot\|_\infty)$ is an isomorphism of normed spaces. Consequently, $\tilde{\mathscr{D}}_*(\tilde{G})$ is an inner product on $\real^k$ that induces $\| \cdot \|_\infty$. On $\integer^k$, $\tilde{\mathscr{D}}_*(\tilde{G}) = G \circ \iota$.
\end{proof}

By the following lemma, when (i)-(iv) hold, the involutive $C^{r-1}$ distribution $\cap_{i=1}^k \hat{\omega}_{w_i}^\perp$ has constant dimension $n-k$, and by Frobenius's theorem it foliates $\hat{N}$ by $C^r$ $(n-k)$-dimensional submanifolds, each of which is contained in an intersection of the form $\cap_{i=1}^k H_i$ for horospheres $H_1,\ldots,H_k$ of $w_1,\ldots,w_k$, respectively.

\begin{lemma}\label{linearly independent}
    Let $Z$ be a subgroup of $Z(\pi_1(N))$ of rank $k$ for which statements (i)-(iv) of Theorem \ref{linear splitting for subgroups} hold. Then, for each $\hat{x} \in \hat{N}$, the set $\{ \hat{\omega}_{w_1}(\hat{x}),\ldots,\hat{\omega}_{w_k}(\hat{x}) \}$ is linearly independent.
\end{lemma}

\begin{proof}
    Suppose that, for some $c_1,\ldots,c_k \in \real$, $\sum_{i=1}^k c_i \omega_{w_i}(\hat{x}) = 0$. Then
    \begin{align*}
        0 &= \Big\| \sum_{i=1}^k c_i \omega_{w_i}(\hat{x}) \Big\|_h^2\\
        &= \sum_{i,j=1}^k c_i c_j h \big( \omega_{w_i}(\hat{x}),\omega_{w_j}(\hat{x}) \big)\\
        &= \sum_{i,j=1}^k c_i c_j h \big( \omega_{w_i}(\hat{y}),\omega_{w_j}(\hat{y}) \big)\\
        &= \Big\| \sum_{i=1}^k c_i \omega_{w_i}(\hat{y}) \Big\|_h^2
    \end{align*}
    for all $\hat{y} \in N$. Thus $\sum_{i=1}^k c_i \omega_{w_i} = 0$ on $N$. Therefore, $0 = \|\sum_{i=1}^k c_i \omega_{w_i}\|_H = \|\sum_{i=1}^k c_i \omega_{w_i}\|_G$ and, consequently, $0 = \mathscr{D}(\sum_{i=1}^k c_i \omega_{w_i}) = \sum_{i=1}^k c_i e_i$. This forces $c_i = 0$ for all $i$.
\end{proof}

\noindent The next lemma implies that each intersection $\cap_{i=1}^k H_i$ is contained in one leaf of the foliation and, as a consequence, its leaves are exactly those intersections.

\begin{lemma}\label{intersection of two horospheres is connected}
    Let $Z$ be a subgroup of $Z(\pi_1(N))$ for which statements (i)-(iv) in Theorem \ref{linear splitting for subgroups} hold. Fix $\hat{x} \in \hat{N}$, and, for each $i$, denote by $H_i$ the horosphere of $w_i$ through $\hat{x}$. Then $\hat{H} = \cap_{i=1}^k H_i$ is connected.
\end{lemma}

\begin{proof}
    The proof is by induction. Write $v_j = \hat{\omega}_{w_j(\hat{x})}$. Since $b_{v_1}$ has nonzero gradient, the projection $\hat{N} \to \hat{H}_{z_1}$ along the integral curves of $b_{v_1}$ is a continuous surjection, which implies that $H_1$ is connected. If the result holds for $\cap_{i=1}^j H_i$, then, by Lemma \ref{constant inner products}(c), the restriction of $b_{v_{j+1}}$ to $\cap_{i=1}^j H_i$ has nonzero gradient, so $\cap_{i=1}^{j+1} H_i$ is similarly connected.
\end{proof}

\noindent Define a map $\Psi : \hat{H} \times \real^k \to \hat{N}$ in the following way: For each $(\hat{y},s_1,\ldots,s_k) \in \hat{H} \times \real^k$, let $\hat{x}_0 = \hat{y}$, and inductively define $\hat{x}_{i+1} = \gamma_{\hat{\omega}_{i+1}(\hat{x}_i)}(s_{i+1}) = \exp_{\hat{x}_i}(s_{i+1} \hat{\omega}_{i+1}(\hat{x}_i))$. Let $\Psi(\hat{y},s_1,\ldots,s_k) = \hat{x}_k$. The splittings $\hat{N} \cong H_i \times \real$ ensure that $D \Psi$ is nonsingular, and, consequently, the inverse function theorem implies that $\Psi$ is a local diffeomorphism.

\begin{lemma}
    $\Psi$ is proper.
\end{lemma}

\begin{proof}
    It follows from Corollary \ref{integrate with respect to h} that $[B(w_i,w_j)]$ is a positive-definite symmetric matrix, so $\sum_{i=1}^k |b_i \circ \Psi(\hat{y},m_1,\ldots,m_k)| \to \infty$ uniformly as $\sum_{i=1}^k |m_i| \to \infty$ for $m_i \in \integer$. By periodicity, there exists a Lipschitz constant, uniform in $\hat{y} \in \hat{N}$, for all maps of the form $\Psi(\hat{y},\cdot)$. Thus $\sum_{i=1}^k |b_i \circ \Psi(\hat{y},s_1,\ldots,s_k)| \to \infty$ uniformly as $\sum_{i=1}^k |s_i| \to \infty$. Let $X \subset \hat{N}$ be compact. Then $\sum_{i=1}^k |b_i|$ is bounded on $X$, which implies that the projection of $\Psi^{-1}(X)$ onto the $\real^k$-factor is compact. At the same time, the projection of $\Psi^{-1}(X)$ onto the $\hat{H}$-factor is contained within a closed ball around $X$, so it too is compact. Thus $\Psi^{-1}(X)$ is compact.
\end{proof}

\noindent Hadamard's global inverse function theorem implies that $\Psi$ is a diffeomorphism, which proves parts (c) and, in turn, (b) of Theorem \ref{splitting for subgroups when integrals vanish}.

For a fixed $\hat{x} \in \hat{N}$, let $\hat{H}$ be the intersection of horospheres $\cap_{i=1}^k H_i$ containing $\hat{x}$. Following the argument in \cite{O'Sullivan1974}, set $G_0 = \{ g \in \pi_1(N) \st g(\hat{x}) \in \hat{H} \}$. By Lemma \ref{constant change in busemann functions}, $G_0$ is normal, contains the commutator subgroup $[\pi_1(N),\pi_1(N)]$, and acts freely and properly discontinuously on $\hat{H}$ by isometries. Note that the subgroup $G_0'$ of $G_0$ consisting of orientation-preserving elements has all of those same properties, and that the quotient space $N_0 = \hat{H} / G_0'$ is orientable. The subgroup $G$ generated by $G_0'$ and $Z$ is isomorphic to $G_0' \times Z$, and the quotient space $\hat{H} / G$ is diffeomorphic to $N_0 \times \tor^k$. One obtains normal covering maps
\[
    \hat{H} \times \real^k \xrightarrow{\psi_0} N_0 \times \tor^k \xrightarrow{\phi_0} N\textrm{.}
\]
Note that $\psi_0$ may be assumed a product and that $G_0' = \pi_1(N_0)$. Let $\Gamma$ denote the quotient group $\pi_1(N) / G$. By construction,
\[
    0 \to \pi_1(N_0) \times Z \to \pi_1(N) \to \Gamma \to 0
\]
is a short exact sequence, which implies that
\[
    0 \to (\pi_1(N_0)/[\pi_1(N),\pi_1(N)]) \times Z \to H_1(N,\integer) \to \Gamma \to 0
\]
is as well. This proves Theorem \ref{splitting for subgroups when integrals vanish}(d).

\section{Proof of Theorem \ref{convex splitting for subgroups}}

A straightforward volume argument shows that, whenever a central Busemann function is everywhere subharmonic or everywhere superharmonic, it must be harmonic.

\begin{lemma}\label{central busemann functions are harmonic}
    Let $N$ be a compact Riemannian manifold with no conjugate points and $b_v$ a central Busemann function on $\hat{N}$. If $b_v$ is either sub- or superharmonic, then it is harmonic.
\end{lemma}

\begin{proof}
    If $b_v$ vanishes identically, the result is clear. Suppose $b_v$ corresponds to a nontrivial element of $Z(\pi_1(N))$. Let $T$, $U$, and, for each $t \in \real$, $H_t$ be as in Lemma \ref{derivative of area}. Then
    \[
        \frac{d}{dt} \vol_{H_{T+t}}(U \times \{ T + t \}) = \frac{1}{\|z\|_\infty} \int_{U \times \{ T + t \}} \Delta b_v \, d\vol_{H_{T+t}}\textrm{.}
    \]
    Since the right-hand side is either nonnegative or nonpositive, $\vol_{H_{T+t}}(U \times \{ T + t \})$ is $\|z\|_\infty$-periodic, and $U$ and $T$ are arbitrary, $\Delta b_v$ must vanish identically.
\end{proof}

\noindent If $b_v$ is a convex or concave central Busemann function, then it is subharmonic or, respectively, superharmonic. By Lemma \ref{central busemann functions are harmonic}, it is therefore harmonic. Since every convex or concave harmonic function has vanishing Hessian, one obtains the following.

\begin{lemma}\label{central busemann functions are totally geodesic}
    Let $b_v$ be a central Busemann function on $\hat{N}$. If $b_v$ is convex or concave, then it is totally geodesic.
\end{lemma}

\noindent For the remainder of this section, the hypotheses of Theorem \ref{convex splitting for subgroups} will be assumed. In this case, the conclusions of Theorem \ref{linear splitting for subgroups} hold. As before, let $w_1,\ldots,w_k$ generate $Z$.

Fix $\hat{x} \in N$, and, for each $i$, let $v_i = \hat{\omega}_{w_i}(\hat{x})$. Since $b_{v_i}$ is totally geodesic, each of its horospheres is a totally geodesic submanifold of $\hat{N}$. This and the fact that $\nabla_{\hat{\omega}_{w_i}} \hat{\omega}_{w_i} = 0$ imply that $\hat{\omega}_{w_i}$ is parallel. Thus the $k$-dimensional distribution $\mathscr{S} = \mathrm{span}\,\{ \hat{\omega}_{w_1}, \ldots, \hat{\omega}_{w_1} \}$ is involutive and, by Frobenius's theorem, foliates $\hat{N}$ by $k$-dimensional submanifolds. Because $\mathscr{S}$ restricts on each leaf of this foliation to a globally parallel orthonormal frame, its leaves are flat and totally geodesic Euclidean spaces. It follows from de Rham's splitting theorem \cite{deRham1952} that $\hat{N}$ is isometric to $\hat{H} \times \real^k$ for any intersection of horospheres $\hat{H}$ as in Theorem \ref{linear splitting for subgroups}. Note that $Z$ acts on each $\real^k$-fiber by translations, which completes the proof of Theorem \ref{convex splitting for subgroups}(a). The proof of part (b) follows exactly as in \cite{O'Sullivan1974}: The distribution $\mathscr{D}$ projects to a parallel distribution on $N$, the leaves of which are compact, flat, totally geodesic, and without holonomy; consequently, they must be toruses.

Unlike in the case of nonpositive curvature, it is not known that a nontrivial subgroup of the fundamental group of a compact manifold with no conjugate points has center consisting of its Clifford translations (cf. Proposition 2.3 of \cite{ChenEberlein1980} and Lemma 3 of \cite{Eberlein1982}). However, Theorem \ref{convex splitting for subgroups}(c) may still be proved using Theorem \ref{virtual splitting} and the argument in Remark 1 of \cite{Eberlein1982}. Corollary \ref{virtual splitting for higher rank} implies that $N$ is finitely covered by a manifold $\tilde{N}$ with $\pi_1(\tilde{N}) \cong G' \times \integer^k$ for a subgroup $G'$ of $\pi_1(N)$, where $w_1^{m_1},\ldots,w_k^{m_k}$ generate the $\integer^k$-factor. Without loss of generality, one may replace $G'$ with its orientation-preserving elements. By Theorem \ref{convex splitting for subgroups}(a) and Lemma \ref{constant change in busemann functions}, elements of $\pi_1(N)$ preserve the horizontal and vertical foliations of $\hat{H} \times \real^k$. Thus $G'$ acts on $\hat{H}$. If this action were not free and properly discontinuous, one could construct, by an argument similar to the proof of Lemma \ref{primitive elements}, elements of $\pi_1(\tilde{N})$ of arbitrarily small displacement, contradicting the compactness of $N$. Thus the quotient space $N_1 = \hat{H} / G'$ is an orientable manifold.

An elementary argument shows that $\tilde{N}$ has the structure of a $\tor^k$-bundle, with totally geodesic and flat $\tor^k$-fibers, over $N_1$. The proof of Lemma 4 of \cite{Eberlein1982}, except with $C^r$ regularity, passes through exactly as written. Thus one may construct a $C^r$ section of the $\tor^k$-bundle as in Remark 1 of \cite{Eberlein1982}, which then implies that $\tilde{N}$ is diffeomorphic to $N_1 \times \tor^k$ and that the covering map $N_1 \times \tor^k \to N$ restricts on each $\tor^k$-fiber to a totally geodesic and locally isometric immersion onto a leaf of the $\tor^k$-foliation of $N$. This proves part (c).

\section{Additional results and questions}

This work was largely motivated by the question of whether the center theorem of Wolf \cite{Wolf1964} and O'Sullivan \cite{O'Sullivan1974} generalizes to the case of no conjugate points.

\begin{question}\label{question1}
    Let $N$ be a compact Riemannian manifold with no conjugate points with $Z(\pi_1(N))$ of rank $k$.\\
    (a) Is $N$ foliated by totally geodesic and flat $k$-toruses?\\
    (b) Does the universal covering space $\hat{N}$ split isometrically as $\hat{H} \times \real^k$?
\end{question}

\noindent As a starting point, one might try to show that the asymptotic norm on $Z(\pi_1(N))$ is Riemannian.

It is natural to look for geometric conditions, weaker than having no focal points, that ensure the linear splitting of central Busemann functions in Theorem \ref{linear splitting for subgroups}. The following argument shows that it suffices to control the asymptotic geometry of distance spheres on $\hat{N}$: Fix a nontrivial $z \in Z(\pi_1(N))$. The $C^{r-1}$ regularity of $\omega_z$ guarantees that the stable Jacobi tensor along each $\gamma_{\hat{\omega}_z(\hat{x})}$ is bounded. The argument in Proposition 5 of \cite{Eschenburg1977} shows that, at each point $\hat{x} \in \hat{N}$, the second fundamental forms of the distance spheres $\partial B(\gamma_v{z}(-t),t)$, for $v = \hat{\omega}_z(\hat{x})/\|\hat{\omega}_z(\hat{x})\|$, converge to that of the horosphere of $z$ through $\hat{x}$.

For any unit vector $v \in T_{\hat{x}} \hat{N}$, define the distance function $\rho_{\hat{x}}(\cdot) = d_{\hat{N}}(\cdot,\hat{x})$ to be \textbf{asymptotically subharmonic in the direction of $v$} (respectively, \textbf{superharmonic}) if $\liminf_{t \to \infty} \Delta \rho_{\hat{x}}(\gamma_v(t)) \geq 0$ (respectively, $\limsup_{t \to \infty} \Delta \rho_{\hat{x}}(\gamma_v(t)) \leq 0$). On $\hat{N} \setminus \{ \hat{x} \}$, let $\kappa_{\hat{x}}^-$ and $\kappa_{\hat{x}}^+$ be the functions equal to the smallest and, respectively, largest eigenvalues of $\mathrm{Hess}\,\rho_{\hat{x}}$, and similarly define $\rho_{\hat{x}}$ to be \textbf{asymptotically convex in the direction of $v$} (respectively, \textbf{concave}) if $\liminf_{t \to \infty} \kappa_{\hat{x}}^-(\gamma_v(t)) \geq 0$ (respectively, $\limsup_{t \to \infty} \kappa_{\hat{x}}^+(\gamma_v(t)) \leq 0$). Lemmas \ref{central busemann functions are harmonic} and \ref{central busemann functions are totally geodesic} imply the following.

\begin{proposition}
    Let $z \in Z(\pi_1(N))$. Then the following hold:\\
    (a) If, for all $\hat{x} \in \hat{N}$, $\rho_{\hat{x}}$ is asymptotically subharmonic in the direction of $v = \hat{\omega}_z(\hat{x})$, then, for each such $v$, $b_v$ is harmonic;\\
    (b) If, for all $\hat{x} \in \hat{N}$, $\rho_{\hat{x}}$ is asymptotically convex in the direction of $v = \hat{\omega}_z(\hat{x})$, then, for each such $v$, $b_v$ is totally geodesic.
\end{proposition}

\noindent Similar results hold for asymptotically superharmonic or asymptotically concave central Busemann functions.

\begin{question}\label{question2}
    Are there natural geometric conditions, other than having no focal points, that ensure central Busemann functions are asymptotically subharmonic or asymptotically convex?
\end{question}

\noindent Natural conditions that have been considered are Ricci curvature bounds, the effects and limitations of which are discussed thoroughly in \cite{EschenburgO'Sullivan1980}.

It's also unclear whether the Heber foliation of the unit sphere bundle of a torus with no conjugate points \cite{Heber1994} generalizes to the case where Busemann functions in $Z$ split linearly.

\begin{question}\label{question3}
    Let $N$ be a compact Riemannian manifold with no conjugate points and $Z$ a subgroup of $Z(\pi_1(N))$ such that statements (i)-(iv) in Theorem \ref{linear splitting for subgroups} hold. Write $\mathscr{W} = \{ \sum_{i=1}^k a_i \hat{\omega}_{w_i} \st a_i \in \real \}$, and, for any $\hat{x} \in \hat{N}$, write $\mathscr{V}_{\hat{x}} = \{ \nabla b_v \st v = \sum_{i=1}^k a_i \hat{\omega}_{w_i}(\hat{x}) \textrm{ for } a_i \in \real \}$. Is $\mathscr{V}_{\hat{x}} = \mathscr{W}$?
\end{question}

\noindent An elementary argument shows that, for each nonzero $a = (a_1,\ldots,a_k) \in \real^k$, there is a unique $C^r$ horofunction $h_a$ such that any sequence of rational directions that converge to $a$ induces a sequence of Busemann functions that vanish at $\hat{x}$ and converge uniformly on compact sets to $h_a$. Moreover, the gradient flow of $h_a$ is through geodesics, and its level sets are the integral submanifolds of the codimension-one involutive distribution $(\sum_{i=1}^k a_i \hat{\omega}_{w_i})^\perp$. However, without the linear divergence of geodesics in \cite{Heber1994}, it's not clear that, when $a$ is irrational, $h_a = b_v$ for $v = \sum_{i=1}^k a_i \hat{\omega}_{w_i}(\hat{x})$.

\bibliography{bibliography}
\bibliographystyle{amsplain}

\end{document}